\documentclass[12pt]{amsart}
\usepackage{amssymb,verbatim,amscd,amsmath,graphicx,enumerate}
\usepackage{amsmath}
\usepackage{graphicx}
\usepackage{caption}
\usepackage{subcaption}
\usepackage{pdfsync}
\usepackage{fullpage}
\usepackage{color}
\usepackage{enumitem}
\usepackage{marginnote}
\usepackage{tikz}

 \newenvironment{dedication}
        {\vspace{6ex}\begin{quotation}\begin{center}\begin{em}}
        {\par\end{em}\end{center}\end{quotation}}

\usetikzlibrary{patterns}
\usetikzlibrary{calc}
\usetikzlibrary{matrix}
\usepgflibrary{arrows}
\usetikzlibrary{arrows}
\usetikzlibrary{decorations.pathreplacing}
\usetikzlibrary{decorations.pathmorphing}
\usepackage{hyperref}

\numberwithin{equation}{section}
\newtheorem{theorem}{Theorem}[section]
\newtheorem{lemma}[theorem]{Lemma}

\newtheorem{corollary}[theorem]{Corollary}

\theoremstyle{definition}
\newtheorem{definition}[theorem]{Definition}

\newtheorem{def-prop}[theorem]{Definition-Proposition}
\newtheorem{remark}[theorem]{Remark}
\newtheorem{example}[theorem]{Example}

\newtheorem*{Mysketch}{Sketch of proof} 
  {\pushQED{\qed}\begin{Mysketch}}
  {\popQED\end{Mysketch}}

\DeclareMathOperator{\Ass}{Ass}
\DeclareMathOperator{\Min}{Min}

\DeclareMathOperator{\depth}{depth}

\DeclareMathOperator{\ini}{in_{>}}
\DeclareMathOperator{\pol}{pol}

\def\NN{\mathbb{N}}

\def\G{\mathcal{G}}

\def\P{\mathcal{P}}

\newcommand{\lcm}[1]{\ensuremath{{\rm{lcm}}{#1}}}

\begin{document}

\title{Regular sequences on squares of monomial ideals}

\author{Louiza Fouli}
\address{Department of Mathematical Sciences \\
New Mexico State University\\
P.O. Box 30001 \\
Department 3MB \\
Las Cruces, NM 88003}
\email{lfouli@nmsu.edu}
\urladdr{http://www.web.nmsu.edu/~lfouli}

\author{T\`ai Huy  H\`a}
\address{Department of Mathematics \\
Tulane University \\
6823 St. Charles Avenue \\
New Orleans, LA 70118}
\email{tha@tulane.edu}
\urladdr{http://www.math.tulane.edu/~tai/}

\author{Susan Morey}
\address{Department of Mathematics \\
Texas State University\\
601 University Drive\\
San Marcos, TX 78666}
\email{morey@txstate.edu}
\urladdr{https://www.math.txstate.edu/about/people/faculty/morey.html}


\begin{abstract}
We use initially regular sequences that consist of linear sums to explore the depth of $R/I^2$, when $I$ is a monomial ideal in a polynomial ring $R$.
We give conditions under which these linear sums form regular or initially regular sequences on $R/I^2$. We then obtain a criterion for when $\depth R/I^2>1$ and a lower bound on $\depth R/I^2$.
\end{abstract}

\maketitle

\begin{dedication}
\vspace{-1cm}
{Dedicated to Rafael H. Villarreal on the occasion of his 70th birthday.}
\end{dedication}

\section{Introduction}\label{intro}

Let $R$ be a polynomial ring over an arbitrary field $k$ and let $I \subseteq R$ be a homogeneous ideal. In \cite{FHM}, we introduced a new notion, which we shall recall later, of an \emph{initially regular sequence} with respect to $I$, whose length gives a lower bound for the depth of $R/I$.  Regular sequences are quite difficult to produce in practice. Using initially regular sequences one can reduce to the monomial case and use combinatorial structures associated to monomial ideals to determine a single regular element, and a sequence that approximates a regular sequence. Moreover, a focus of \cite{FHM} is determining initially regular sequences directly from a single combinatorial structure associated to the first initial ideal. Using the combinatorics of the first initial ideal, we gave a concrete construction of specific sequences of linear forms that are initially regular and discussed situations in which these sequences are also regular. The use of initially regular sequences allows us to provide a combinatorial lower bound for the depth of the ideal as well as to determine concrete regular sequences in special cases. In this paper, we extend the study in \cite{FHM} a step further and examine when the constructed sequences of linear forms are initially regular or regular sequences on $R/I^t$ with $t\ge 2$. This leads to a criterion for when $\depth R/I^2>1$, partially answering a question raised by Terai and Trung in \cite{TT}.

To be more precise, let us first recall the notion of initially regular sequences from \cite{FHM}.

\begin{definition}[\cite{FHM}]
		Given a term order $>$, a sequence of nonconstant polynomials $f_1, \dots, f_s$ is said to be an \emph{initially regular sequence} on $R/I$ if for each $i = 1, \dots, s$, $f_i$ is a regular element on $R/I_i$, where $I_1 = \ini(I)$ and $I_i = \ini(I_{i-1}, f_i)$ for $i>1$.
\end{definition}	

In this paper, we focus on a particular setting where each linear form in the sequence associated to $I$ is a sum of the form $f = b_0 + b_1 + \dots + b_t$, where $b_i$ are variables in the polynomial ring, $b_0 > b_1> \ldots > b_t$ and
\begin{center}
\begin{minipage}[c]{0.1\textwidth}
$(\star)$ $\left\{ \begin{array}{llll} \\ \\ \\ \\ \end{array}\right.$
\end{minipage}
\hspace*{-1cm}
\begin{minipage}[c]{0.9\textwidth}
\begin{enumerate}
    \item the maximum power of $b_i$ appearing in the minimal generators of $\ini(I)$ is at most 1 for all $0 \le i \le t$; and
    \item if $M$ is a minimal generator of $\ini(I)$ and $b_0$ divides $M$, then there exists a $i \ge 1$ such that $b_i$ divides $M$.
\end{enumerate}
\end{minipage}
\end{center}

Condition $(\star)$ stems from the observation that if you select any vertex of a graph $G$ and add to it all its neighbors, the resulting set of vertices will not be contained in any minimal vertex cover, and thus the sum will be regular on $R/I$ when $I$ is the edge ideal associated to $G$. To better understand this definition, consider the following example. 

\begin{example}
Let $R=\mathbb{Q}[a,b,c,d,e,f,g]$ and let $I=(abc, acd, ae, ef, fg)$ be the ideal corresponding to the hypergraph $G$ depicted below.

		
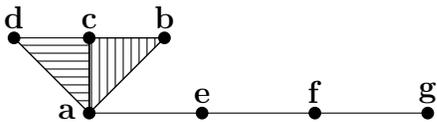
\begin{tikzpicture}
			
			\tikzstyle{point}=[circle,thick,draw=black,fill=black,inner sep=0pt,minimum width=4pt,minimum height=4pt]

			\node (a)[point, label={[xshift=-0.3cm, yshift=-0.3 cm] $\bf{a}$}] at (1,1) {};
			
			\node (b)[point,label={[xshift=0.0cm, yshift=-0.1cm] $\bf{b}$}] at (2,2) {};
			
			\node (c)[point,label={[xshift=-0.0cm, yshift=-0.1cm] $\bf{c}$}] at (1,2) {};
			
			\node (d)[point,label={[xshift=0.0cm, yshift=-0.1cm] $\bf{d}$}] at (0,2) {};
			
			\node (e)[point,label={[xshift=0.0cm, yshift=-0.1cm] $\bf{e}$}] at (2.5,1) {};
			
			\node (f)[point,label={[xshift=0.0cm, yshift=-0.1cm] $\bf{f}$}] at (4,1) {};
			
			\node (g)[point,label={[xshift=0.0cm, yshift=-0.1cm] $\bf{g}$}] at (5.5,1) {};
					
			\draw (a.center) -- (d.center);
			\draw (a.center) -- (e.center);
			\draw (e.center)--(g.center);
			\draw (a.center) -- (b.center);
			\draw (a.center) -- (c.center);
			\draw (c.center) -- (b.center);
			
			\draw[pattern=vertical lines] (a.center) -- (b.center) -- (c.center) -- cycle;
			
			\draw[pattern=horizontal lines] (a.center) -- (d.center) -- (c.center) -- cycle;
			
		\end{tikzpicture}
		
Fix the term order $g>f>d>a>b>c>e$, for instance. The elements $a+c+e$, $d+a$, and $g+f$ each satisfy condition $(\star)$. Note that condition $(1)$ of $(\star)$ holds automatically since $I$ is a square-free monomial ideal. For condition $(2)$, using $a$ in the role of $b_0$, note that either $c$ or $e$ divides each of three generators divisible by $a$. Using $g$ in the role of $b_0$ requires $f$ to cover the single generator divisible by $g$. It can be seen from results in \cite{FHM} that $g+f, d+a, a+c+e$ forms a regular and initially regular sequence on $R/I$. Moreover, on can check using a computer system such as Macaulay2 \cite{M2}, that $\depth(R/I)=3$.
\end{example}

In the study of the depth function, it is desirable to know the depth of powers of an ideal instead of just the depth of the ideal itself (cf. \cite{BHH2014, FHM2020, FM, HNTT, HH2005, HV2013, Morey, Trg2016}). This motivates the following natural question: when does a linear form or a sequence of linear forms constructed in \cite{FHM} remain regular or initially regular with respect to powers of $I$? We shall address this question; particularly, we shall investigate when a linear form $f = b_0 + \dots + b_t$ satisfying $(\star)$, or a sequence of such forms, is regular or initially regular with respect to powers of $I$.

When $I$ is the edge ideal of a graph, a sum $f = b_0 + \dots + b_t$ satisfying condition $(\star)$ corresponds to a \emph{star} subgraph of the given graph, consisting of a vertex and all its neighbors. Building on previous work of Chen, Morey, and Sung \cite{CMS}, our first result gives a satisfactory answer to the aforementioned question for edge ideals of a large class of graphs, including bipartite graphs (containing no odd cycles) and those having a unique odd cycle. We show that when the vertices $b_0, \dots, b_t$ are far enough from the unique odd cycle (if one exists) in the graph, the linear form $f = b_0 + \dots + b_t$ is regular with respect to powers of $I$.

\medskip

\noindent \textbf{\autoref{thm.graph}.} Let $G = (V_G,E_G)$ be a graph and let $I = I(G)$ be its edge ideal. Suppose that $f = b_0+\dots +b_t$ is a sum satisfying condition $(\star)$.
\begin{enumerate}
	\item If $G$ is bipartite, then $f$ is regular on $R/I^r$ for all $r \in \NN$.
	\item Suppose that $G$ contains a unique odd cycle $C$ of length $2k+1$ and $\ell = d(b_0,C)$ is the distance from $b_0$ to $C$. Then $f$ is regular on $R/I^r$ for all $r \le k+\ell -1$.
\end{enumerate}

For monomial ideals in general, the question of when such a linear form $f = b_0 + \dots +b_t$ is regular with respect to powers of $I$ turns out to be very difficult. The bulk of our work then restricts to $I^2$ and the case where $f$ has a small number of terms; particularly, when $f$ is a binomial or a trinomial.

In \cite{FHM}, the question of when $f=b_0+b_1$ is regular with respect to $R/I$ was examined. We investigate when $f=b_0+b_1$ is regular with respect to $R/I^2$, where $I$ is any monomial ideal. Specifically, we prove the following theorem.

\medskip

\noindent \textbf{\autoref{sum 2 regular on power}.} Let $I$ be a monomial ideal in a polynomial ring $R$.
Let $b_0$ and $b_1$ be distinct variables that satisfy condition $(\star)$. Suppose that there do not exist variables $c$ and $d$ such that all three $b_1c, b_1d$ and $cd$ divide minimal generators of $I$. Then, $f = b_0+b_1$ is regular on $R/I^2$.

\medskip

We extend this result in Theorem~\ref{sequences for square} to give a criterion for when linear forms with disjoint support that satisfy condition $(\star)$ can be combined together to form a regular sequence with respect to $I^2$.

We then address the question of when trinomial linear sums of the form $f = b_0+b_1+b_2$ form initially regular sequences on $R/I^2$. It turns out that the question for trinomial linear sums is much more complicated than that for binomial linear sums. We give a criterion in Theorem~\ref{sequence sum 3} for when a sequence  of such trinomial linear sums is initially regular with respect to the second power of a monomial ideal. This further leads us to a result that says we may combine the two types of sequences from Theorems~\ref{sequences for square} and~\ref{sequence sum 3} to form longer initially regular sequences on $R/I^2$, Corollary~\ref{combine}.

\section*{Acknowledgements}
The authors would like to thank the referees for the careful review of the article and for the suggestions that improved the article. The second author is partially supported by Simons Foundation and Louisiana Board of Regents.

\section{Powers of edge ideals of graphs}

In this section, we address powers of edge ideals of graphs, focusing particularly on the class of bipartite graphs, i.e., graphs containing no odd cycles, and those that are close to being bipartite, i.e., those having a unique odd cycle. We shall also provide a number of auxiliary results for monomial ideals in general. Recall that $I$ is the edge ideal of a graph $G$ if $I$ is generated by $x_ix_j$ whenever $\{x_i,x_j\}$ is an edge of $G$, see \cite{Vil}. For general background on monomial ideals and depth see \cite{Vilbook}.

It is known (see \cite{SVV}) that $\depth(R/I^r) \ge 1$, for all $r \in \NN$, if $I = I(G)$ is the edge ideal of a bipartite graph $G$. Hence for every $r \in \NN$, there exists a regular element on $R/I^r$. We provide a means, starting with any vertex of the graph, to produce such a regular element of $R/I(G)^r$, for any $r \in \NN$ (or for sufficiently small $r$ when $G$ contains a unique cycle). For a subgraph $H$ and a vertex $x$ of $G$, we denote by $d(x,H)$ the \emph{distance} between $x$ and $H$, that is the minimum length of a path connecting $x$ and a vertex in $H$ (with the convention that $d(x,H) = \infty$ if no such path exists). If $W \subseteq V_G$, set $N[W]$ to be the closed neighborhood of $W$, which consists of the vertices in $W$ together with all vertices of $G$ adjacent to a vertex in $W$. Similarly, for $W\subseteq V_G$, the open neighborhood of $W$ is $N(W)=N[W]\setminus \{W\}$.

\begin{theorem} \label{thm.graph}
	Let $G = (V_G, E_G)$ be a graph and let $I = I(G)$. Let $b_0 \in V_G$ and suppose that $N(b_0) = \{b_1, \dots, b_t\}$. Set $f = b_0 + \dots + b_t$.
	\begin{enumerate}
		\item If $G$ is a bipartite graph, then $f$ is regular on $R/I^r$ for all $r \in \NN$.
		\item Suppose that $G$ contains a unique odd cycle $C$ of length $2k+1$ and $\ell = d(b_0,C)$. Then $f$ is regular on $R/I^r$ for all $r \le k+\ell-1$.
	\end{enumerate}
\end{theorem}

\begin{proof} (1) It is well known (cf. \cite{SVV}) that if $G$ is bipartite then $\Ass(R/I^r) = \Min(R/I)$ for all $r \in \NN$. Moreover, by \cite[Theorem 3.11]{FHM}, $f$ is regular on $R/I$. Hence, $f$ is regular on $R/I^r$ for all $r \in \NN$.
	
\noindent(2) It suffices to show that for $r \le k+\ell-1$, $f \not\in \wp$ for all $\wp \in \Ass(R/I^r)$. By \cite[Theorems 4.1 and 5.6]{CMS}, we have $\Ass(R/I^r) = \Min(R/I)$ for $r \le k$ and $\Ass(R/I^r) = \Min(R/I) \cup \P_r$ for $r > k$, where $\P_r$ consists of prime ideals of the form $(R_r, B_r, W)$, in which $R_r$ is a collection of vertices in $G$ that is constructed recursively starting with the odd cycle, $B_r = N[R_r] \setminus R_r$ and $W$ is a minimal subset of the vertices in $G$ such that $R_r \cup B_r \cup W$ forms a vertex cover of $G$.
	
It follows from \cite[Theorem 3.11]{FHM} that $f$ is regular on $R/I$. Thus, $f$ is regular on $R/I^r$ for all $r \le k$.
For $r > k$, consider any $\wp = (R_r, B_r, W) \in \P_r$. Notice that, by construction,
	$$\max\{d(z,C) ~\big|~ z \in R_r\} \le r-k-1 \text{ and } \max\{d(z,C) ~\big|~ z \in B_r\} \le r-k.$$
	Therefore, for $r \le k+\ell-1$, we have $\max\{d(z,C) ~\big|~ z \in R_r \cup B_r\} \le \ell-1 < d(b_0,C)$. This implies that $R_r \cup B_r$ does not contain $b_0$. Hence, having $N[b_0] \subseteq R_r \cup B_r \cup W$ would violate the minimality of $W$, since if $R_r \cup B_r \cup W$ is a vertex cover of $G$ containing $N[b_0]$ then $R_r \cup B_r \cup (W\setminus \{b_0\})$ is also a vertex cover of $G$. That is, $N[b_0] \not\subseteq R_r \cup B_r \cup W$. This implies, since $\wp$ is a monomial prime ideal, that $f \not\in \wp$. The assertion follows.
\end{proof}

To illustrate this result, we consider the following example.

\begin{example}
	Let $R=\mathbb{Q}[a,b,c,d,e,f,g,h]$ and let $I=(ab,bc,cd,de,ae,af,fg,gh)$ be the ideal corresponding to the graph $G$ depicted below.

	
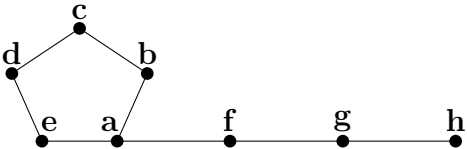
\begin{tikzpicture}
		
		\tikzstyle{point}=[circle,thick,draw=black,fill=black,inner sep=0pt,minimum width=4pt,minimum height=4pt]

		\node (a)[point, label={[xshift=-.1cm, yshift=-0.1cm] $\bf{a}$}] at (1,1) {};
		
		\node (b)[point,label={[xshift=0.0cm, yshift=-0.1cm] $\bf{b}$}] at (1.4,1.9) {};
		
		\node (c)[point,label={[xshift=-0.0cm, yshift=-0.1cm] $\bf{c}$}] at (.5,2.5) {};
		
		\node (d)[point,label={[xshift=0.0cm, yshift=-0.1cm] $\bf{d}$}] at (-.4,1.9) {};
		
		\node (e)[point,label={[xshift=0.1cm, yshift=-0.1cm] $\bf{e}$}] at (0,1) {};
		
		\node (f)[point,label={[xshift=0.0cm, yshift=-0.1cm] $\bf{f}$}] at (2.5,1) {};
		
		\node (g)[point,label={[xshift=0.0cm, yshift=-0.1cm] $\bf{g}$}] at (4,1) {};
		
		\node (h)[point,label={[xshift=0.0cm, yshift=-0.1cm] $\bf{h}$}] at (5.5,1) {};
		
		\draw (a.center) -- (b.center);
		\draw (b.center) -- (c.center);
		\draw (c.center)--(d.center);
		\draw (d.center) -- (e.center);
		\draw (a.center) -- (e.center);
		\draw (a.center) -- (f.center);
		\draw (f.center) -- (g.center);
		\draw (g.center) -- (h.center);

	\end{tikzpicture}
	
	The unique cycle has length $5$ so $k=2$. If $h$ is selected to be $b_0$, then $\ell = 3$. Thus $h+g$ is a regular element on $R/I^r$ for all $r \le 4$. Note that $\depth(R/I^5)=0$, so this example shows that the bound in Theorem~\ref{thm.graph} is optimal.
\end{example}

Generalizing Theorem~\ref{thm.graph} to hypergraphs is difficult, partly because the set of associated primes of powers of a square-free monomial ideal in general (the edge ideal of a hypergraph) is significantly more complex and less well-understood (cf. \cite{FHVT, LT}).  However, for the second power of a square-free monomial ideal, the following result of Terai and Trung \cite{TT} completely describes its associated primes. Recall that a hypergraph $H$ is a collection of vertices $V_H$ together with a collection of edges $E_H$, which are subsets of arbitrary size of the vertex set. A hypergraph is {\em simple} if $e_1 \not\subseteq e_2$ for all $e_1,e_2 \in E_H$.

\begin{definition}[{\cite{TT}}]
	\label{def.2sat}
	Let $H = (V_H, E_H)$ be a simple hypergraph.
	\begin{enumerate}
		\item A subset $U \subseteq V_H$ is \emph{decomposable} in $H$ if $U$ can be partitioned into two subsets each of which contains an edge of $H$; otherwise, $U$ is called \emph{indecomposable}.
		\item A subset $U \subseteq V_H$ is said to be \emph{2-saturating} in $H$ if $U$ is indecomposable in $H$ and $U \setminus \{i\}$ is decomposable in $H(i)$, for every vertex $i \in V_H$, where $H(i)$ is the hypergraph over $V_H \setminus \{i\}$ with edge set $\{e \setminus \{i\} ~\big|~ e \in E_H\}$ (including possibly the empty set).
	\end{enumerate}
\end{definition}

For a subset $C \subseteq V_H$ of the vertices in a hypergraph $H$, let $H_C$ be the hypergraph whose edges are $\{e \cap C ~\big|~ e \in E_H\}$ and set
$$P_C = \langle x_i ~\big|~ i \in C\rangle.$$

\begin{theorem}[{\cite[Theorem~3.1]{TT}}]\label{TT thm}
	Let $I = I(H)$ be the edge ideal of a hypergraph $H$. For a subset $C$ of the vertices in $H$, $P_C$ is an associated prime of $I^2$ if and only if $H_C$ has a $2$-saturating set.
\end{theorem}

As an immediate consequence of Theorem~\ref{TT thm}, we obtain the following corollary.

\begin{corollary} \label{linear sum regular on power}
	Let $I = I(H)$ be an edge ideal of a hypergraph in a polynomial ring $R$. Let $b_0, \ldots, b_t$ be distinct variables in $R$ such that
	$\{b_0, \ldots, b_t\}$ is not a subset of any $2$-saturating set in $H$.
	If $f=b_0+\ldots +b_t$, then $f$ is regular on $R/I$ and $R/I^2$.
\end{corollary}

\begin{proof}
	The assertion follows from \cite[Theorem~3.1]{TT}.
\end{proof}

The following remark concerns Corollary~\ref{linear sum regular on power} when $I$ is the edge ideal of a graph.

\begin{remark}
	Let $I = I(G)$ be the edge ideal of a graph $G$. Note that for a graph, $2$-saturating sets of induced subgraphs are determined by the triangles of $G$. Let $b_0, \ldots, b_t$ be distinct variables that satisfy condition $(\star)$. Suppose also that
if the set $\{b_0, \ldots, b_t\}$ intersects any triangle $T$ in $G$, then for some $0 \le i \le t$, $b_i$ is not connected to any vertex in $T$.
Then $f=b_0+\ldots +b_t$ is regular on $R/I$ by \cite[Theorem~3.11]{FHM} and on $R/I^2$ by Theorem~\ref{TT thm}.  	
\end{remark}

\begin{remark}\label{compare Terai-Trung}
	If $I=I(G)$ is the edge ideal of a hypergraph $G$ and the conditions in Corollary~\ref{linear sum regular on power} hold for $b_0, b_1$, then $f=b_0+b_1$ is regular on $R/I^2$. We shall see later, in Theorem~\ref{sum 2 regular on power}, that a similar phenomenon holds for more general monomial ideals.		
\end{remark}

As we have remarked, associated primes of more general monomial ideals are not yet sufficiently well-understood. A new approach is thus needed in order to proceed further for monomial ideals. This will be carried out in the next two sections. Particularly, we shall closely examine the structure of initial ideals and invoke results from \cite{FHM}.

Recall that $S$-resultants are used in Buchberger's algorithm when computing the Gr\"{o}bner basis of an ideal. For general background information on Gr\"obner bases or Buchberger's algorithm, see \cite{AL}. Basically, the $S$-resultant of two polynomials with respect to a fixed a term order is the polynomial formed by canceling the leading terms of each. That is,
$$S(f,g) = \frac{\lcm(\ini(f),\ini(g))}{\ini(f)}f - \frac{\lcm(\ini(f),\ini(g))}{\ini(g)} g$$
where $\ini(f)$ is the initial term of $f$ relative to the fixed term order.

We collect here a few observations on $S$-resultants that we use later for this purpose. The proofs of each of the statements are straightforward based on the definition of an $S$-resultant but are included for completeness.

\begin{lemma}\label{S-resultant tricks}
	Let $R$ be a polynomial ring over a field $k$, $R_2=k[b_0, \ldots, b_t]$ and let $R_1$ be the polynomial ring such that $R=R_1[b_0, \ldots, b_t]$.
	 The following observations about $S$-resultants hold:	
	\begin{enumerate}[label=$($\alph*$)$]
		\item  For $f_1=M_1g_1, f_2=M_2g_2$, with $M_1, M_2$ are monomials in $R_1$ and $g_1, g_2 \in R_2$,
		$$S(f_1,f_2)=\lcm(M_1, M_2)S(g_1,g_2).$$
	
		\item If $M,N$ are monomials and $f\in R$, then
		$$S(Nf,M)=\dfrac{\lcm(M,N \ini(f))}{\ini(f)}(f-\ini(f)).$$
		\vspace{.15in}
	If $\gcd(M, \ini(f))=1$, then this is a multiple of $M$.
	
		\item 	If $p_1, p_2 \in R$ satisfy $\gcd(\ini(p_1), \ini(p_2))=1$, then
		$$S(p_1, p_2)=\ini(p_1)(p_2-\ini(p_2))-\ini(p_2)(p_1-\ini(p_1)).$$

		\item If $M,N$ are monomials and $f\in R$, then
		$$S(Mf, Nf)=0.$$
	\end{enumerate}

\end{lemma}

\begin{proof}
{\bf $(a)$} Noting that $M_1, M_2 \in R_1$ and $g_1,g_2 \in R_2$, and that $R_1$ and $R_2$ are polynomial rings in disjoint sets of variables, we have	\begin{align*}
S(f_1,f_2) &=  \frac{\lcm(\ini(M_1g_1),\ini(M_2g_2))}{\ini(M_1g_1)}M_1g_1 - \frac{\lcm(\ini(M_1g_1),\ini(M_2g_2))}{\ini(M_2g)_2} M_2g_2 \\
 &=  \frac{\lcm{(M_1,M_2)}\lcm(\ini(g_1),\ini(g_2))}{M_1\ini(g_1)}M_1g_1 - \frac{\lcm(M_1,M_2)\lcm(\ini(g_1),\ini(g_2))}{M_2\ini(g)_2} M_2g_2 \\
 &=\lcm(M_1, M_2)S(g_1,g_2).		
\end{align*}

\medskip

{\bf $(b)$} Since $M$ is a monomial,
\begin{align*}
S(Nf,M) &= \frac{\lcm(M,\ini(Nf))}{\ini(f)} f-  \frac{\lcm(M,\ini(Nf))}{M}M\\
&=\dfrac{\lcm(M,N \ini(f))}{\ini(f)}(f-\ini(f)).
\end{align*}

\medskip

{\bf $(c)$} This follows immediately from the definition of an $S$-resultant and the fact that
$$\lcm(\ini(p_1), \ini(p_2)) = \frac{\ini(p_1)\ini(p_2)}{\gcd(\ini(p_1)\ini(p_2))}.$$

\medskip

{\bf $(d)$} By definition,
\begin{align*}
	S(Mf,Nf) &= \frac{\lcm(\ini(Mf),\ini(Nf))}{\ini(Mf)}Mf - \frac{\lcm(\ini(Mf),\ini(Nf))}{\ini(Nf)} Nf\\
	&= \frac{\lcm(M\ini(f),N\ini(f))}{M\ini(f)}Mf - \frac{\lcm(M\ini(f),N\ini(f))}{N\ini(f)} Nf\\
	&= \frac{\lcm(M,N)\ini(f)}{M\ini(f)}Mf - \frac{\lcm(M,N)\ini(f)}{N\ini(f)} Nf=0.
\end{align*}
\end{proof}

\section{Binomial regular elements and squares of monomial ideals} \label{sec.bin}

In this section, we consider monomial ideals in general. As a trade-off for the lost structure compared to edge ideals of hypergraphs, we will focus on linear sums consisting of only two elements.

We start by showing that for low powers of $I$, a crucial colon ideal stays monomial in this setting. For any monomial ideal $I$ let $\mathcal{G}(I)$ denote the set of minimal monomial generators of $I$. For a monomial $M$ and a variable $x$, we denote by $d_x(M)$ the largest power of $x$ that divides $M$.  For a monomial ideal $I$ and a variable $x$, set
$$d_x(I) = \max \{d_x(M) ~\big|~ M \in \G(I) \}.$$

\begin{lemma} \label{monomial colon 2}
Let $I$ be a monomial ideal in a polynomial ring $R$ over a field $k$. Let $b_0, b_1$ be distinct variables in $R$ that satisfy condition $(\star)$.
Then for $t \le 3$, we have
$$(I^t:b_0+b_1)=(I^t:b_0) \cap (I^t:b_1).$$
In particular, $(I^t:b_0+b_1)$ is a monomial ideal for $t \le 3$.
\end{lemma}

\begin{proof}
Note that $(I^t:b_0) \cap (I^t:b_1) \subseteq (I^t:b_0+b_1)$ for all $t$, so we need only prove the other inclusion. Let $f \in (I^t:b_0+b_1)$. Since $I^t$ is monomial and $b_0 + b_1$ is homogeneous, we may assume $f$ is homogeneous. Write $f = \sum_{i=1}^{\ell}c_if_i$, where $f_1, \ldots, f_{\ell}$ are distinct monomials of the same degree and $c_i \in k$ for all $i$. By replacing $f$ by $f-c_if_i$ as needed, we may assume $f_i \not\in I^t$ for all $i$.

If $f b_0 \in I^t$, then $fb_1 = f(b_0+b_1) - fb_0 \in I^t$ as well. So it suffices to show $f \in (I^t:b_0)$.
Now $fb_0 \in I^t$ if and only if $f_ib_0 \in I^t$ for all $i$. Assume by way of contradiction that $f_i b_0 \not\in I^t$ for some $i$ and for some $t\le 3$.  Without loss of generality, assume $f_1$ is maximal with respect to the degree of $b_0$ among terms with this property. That is,
$$d_{b_0}(f_1) = \max_i\{d_{b_0}(f_i) \mid f_i b_0 \not\in I^t\}.$$
Set $s=d_{b_0}(f_1)$.
Since $f_1 b_0 \not\in I^t$, then $f_1 b_0 = f_jb_1$ for some $j$. It follows that $d_{b_0}(f_j)= s +1$. Thus $f_j b_0 \in I^t$. So $f_jb_0 = M_1\cdots M_t h$ for some $M_i\in \G(I)$ and some monomial $h$. Since $f_j \not\in I^t$ and $b_0$ has degree one in $I$, there must be $s+2$
of the $M_i$ that involve $b_0$. By definition, if $b_0 \mid M_i$, then $b_1 \mid M_i$ as well. Thus we have $$ d_{b_1}(f_j b_0) = d_{b_1} (f_j) \ge s+2.$$
Thus $f_j = b_0^{s+1}b_1^{s+2}N$ for some monomial $N$ with $b_0 \nmid N$.  As a result, $d_{b_1}(f_jb_1) \ge s+3$.
In addition, since $f_j b_0 \in I^{t}$, we may assume $s+2 \le t$. In particular, if $t \le 3$, then $s \le 1$.

Now consider $f_1b_1=b_0^{s}b_1^{s+4}N$. First assume $f_1b_1 \not\in I^t$. Then without loss of generality, $f_1b_1 = f_2 b_0$. If $s=0$ this is a contradiction. So suppose $s=1$. Then $f_2 = b_1^5 N$. Now if $f_2b_1 \in I^t$, then since $d_{b_1}(I) \le 1$, we have $t \ge 6$, which is a contradiction to $t \le 3$. But if $f_2b_1 \not\in I^t$, then $f_2b_1=f_kb_0$ for some $k$, which implies $b_0 \mid f_2$, a contradiction.

Now assume $f_1b_1 \in I^t$. We then have $f_1b_1 = M_1' \cdots M_t' h'$ for some $M_i'\in \G(I)$ and some monomial $h'$, where at least $s+4$ of the $M_i'$ involve $b_1$. Since $s \ge 0$, then $t \ge 4$. Thus for $t \le 3$, we have a contradiction. Thus the statement holds for $t \le 3$.

Since $(I^t:b_i)$ is monomial, it follows that for $t \le 3$,  $(I^t: b_0+b_1)$ is the intersection of two monomial ideals and is thus monomial.
\end{proof}

The following example illustrates that Lemma~\ref{monomial colon 2} cannot be extended to higher powers of $I$.

\begin{example}
Let $R=k[a,b,x_1,x_2,x_3,x_4,y_1,y_2,y_3,y_4]$ and set
$$I=(ab, bx_1y_1,bx_2y_2,bx_3y_3,bx_4y_4,x_1x_2x_3x_4,y_1y_2y_3y_4).$$
Set $f_1 = b^3x_1x_2x_3x_4y_1y_2y_3y_4$ and $f_2 = ab^2 x_1x_2x_3x_4y_1y_2y_3y_4$. Then:
\begin{itemize}
	\item $af_1=bf_2\not\in I^4$;
	\item $bf_1 = (bx_1y_1)(bx_2y_2)(bx_3y_3)(bx_4y_4) \in I^4$;
	\item $af_2 = (ab)^2(x_1x_2x_3x_4)(y_1y_2y_3y_4) \in I^4$.
\end{itemize}
Thus $f_1-f_2 \in (I^4:a+b)$ but $f_1, f_2 \not\in (I^4:a+b)$. So $(I^4:a+b)$ is not monomial.
\end{example}

We are now able to provide conditions under which $\depth(R/I^2)$ is positive, along with a regular element, for monomial ideals. Note that if $I$ is an edge ideal, the condition below ensures the element is not in a 2-saturating set, but our theorem holds in a more general setting than \autoref{TT thm} since $I$ need not be square-free.

\begin{theorem} \label{sum 2 regular on power}
Let $I$ be a monomial ideal  in a polynomial ring $R$. Let $b_0, b_1$ be distinct variables in $R$ that satisfy condition $(\star)$. Suppose that there do not exist variables $c,d$, not necessarily distinct, such that all three of $b_1c, b_1d, cd$ divide minimal generators of $I$. Then, $b_0+b_1$ is regular on $R/I^2$.

\end{theorem}

\begin{proof}
Suppose that $\G(I)=\{M_1, \ldots, M_p\}$ and let $f\in R$ be a polynomial such that $f(b_0+b_1) \in I^2$. By Lemma~\ref{monomial colon 2} we may assume that $f$ is monomial. Then $fb_0, fb_1 \in I^2$. Since $fb_0 \in I^2$, without loss of generality, we may write $fb_0=M_1M_2g$ with $g\in R$. If $b_0 \mid g$, then $f\in I^2$ and the assertion follows. Hence, we may assume that $b_0 \mid M_1$. By our assumptions, $b_1 \mid M_1$ and thus $f=b_1M_1'M_2g$, where $M_1=b_0b_1M_1'$. Since $d_{b_i}(I) \le 1$, we have $b_0, b_1 \nmid M_1'$.

Suppose first that $b_1\mid M_2g$. Then $d_{b_1}(fb_1) \ge 3$. This, together with the fact that $fb_1 \in I^2$ and $d_{b_1}(I) \le 1$, implies that $f \in I^2$ and the assertion is proven.
Thus, we may assume that $b_1 \nmid M_1'M_2g$.

It follows from $fb_1=b_1^2M_1'M_2g \in I^2$ and $d_{b_1}(I^2)\le 2$ that either $f \in I^2$, and we are done, or we may write $fb_1=M_kM_rg'$, with $b_1 \mid M_k$ and $b_1\mid M_r$ and $g' \in R$.
Write $M_k=b_1M_k'$ and $M_r=b_1M_r'$.

Observe that $f=b_1M_1'M_2g=b_1M_k'M_r'g'$ and $b_1 \nmid M_1'M_2g$, so we have that $b_1 \nmid M_k'M_r'g'$. As $fb_0 \in I^2$ and $b_1 \nmid M_k'M_r'g'$, by condition $(\star)$ we must have a decomposition
$$fb_0=(b_0b_1M_v')(\dfrac{1}{M_v'})M_k'M_r'g' $$
 with $b_0b_1M_v'=M_v \in I$ and $\dfrac{1}{M_v'}M_k'M_r'g' \in I$. It follows that $M_k'M_r'g' \in I$. Thus, there exists a monomial generator $M_u$ of $I$ such that $M_k'M_r'g' = M_u h$ for some monomial $h$.

Observe that if $N = \gcd(M_k', M_u)$ divides $\gcd(M_r',M_u)$ then, since $M_k'M_r'g = M_uh$, we must have either $N \mid h$ or there is a variable $c$ dividing $N$ such that $c^2$ divides $M_u$. If $c^2$ divides $M_u$ then, together with the fact that $b_1c$ divides $b_1M_k' = M_k$ and $b_1M_r' = M_r$, we get a contradiction to the hypotheses. Suppose that $N \mid h$. This implies that $M_r'g$ is a multiple of $M_u$, and so $M_r'g \in I$. Particularly, we have $f = b_1M_k'M_r'g = M_k M_r'g \in I^2$. Thus, $N\not= 1$ and $N\nmid \gcd(M_r',M_u)$. By a similar argument, we also have $\gcd(M_r',M_u) \not= 1$ and it does not divide $\gcd(M_k',M_u)$. Therefore, there exist variables $c \mid \gcd(M_k', M_u)$ and $d \mid \gcd(M_r', M_u)$ with $c \not= d$. Then, $cd \mid M_u$,
and together with the fact that $b_1c \mid M_k$ and $b_1d \mid M_r$, we again arrive at a contradiction. Hence, $f \in I^2$ and the proof is complete.
\end{proof}

We provide examples to illustrate the conditions of \autoref{sum 2 regular on power}. Note in particular that  $c,d$ need not be distinct.

\begin{example}
 Let $I=(ab,bc,bd,cd,de,ef)$ be an ideal in the polynomial ring $R=\mathbb{Q}[a,b,c,d,e,f]$. Then $I$ is the edge ideal of the graph below.
   
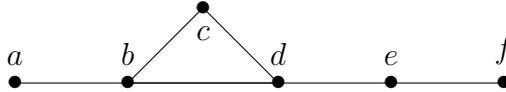
\begin{center}
 	\begin{tikzpicture}
 		\tikzstyle{point}=[inner sep=0pt]
 		\node (a)[point,label=above:$a$] at (-2,0) {$\bullet$};
 		\node (b)[point,label=above:$b$] at (-0.5,0) {$\bullet$};
 		\node (c)[point,label=below:$c$] at (0.5,1) {$\bullet$};
 		\node (d)[point,label=above:$d$] at (1.5,0) {$\bullet$};
 		\node (e)[point,label=above:$e$] at (3,0) {$\bullet$};
 		\node (f)[point,label=above:$f$] at (4.5,0) {$\bullet$};
 		
 		\draw (a.center) -- (b.center);
 		\draw (b.center) -- (d.center);
 		\draw (b.center) -- (d.center);
 		\draw (b.center) -- (c.center);
 		\draw (c.center) -- (d.center);
 		\draw (d.center) -- (e.center);
 		\draw (e.center) -- (f.center);
 	
 	\end{tikzpicture}
 \end{center}
One can see using Macaulay~2~\cite{M2} for example that $\depth R/I=2$ and $\depth R/I^2=1$. In particular, $a+b, f+e$ is a regular sequence on $R/I$. By Theorem~\ref{sum 2 regular on power} we have that $f+e$ is regular on $R/I^2$. Moreover, $a+b$ is not regular on $R/I^2$. Note that $a+b$ violates the condition needed in Theorem~\ref{sum 2 regular on power}.
\end{example}

\begin{example}
Let $I=(a^2b,bc^2,bd^2,cd,de,ef)$ be an ideal in the polynomial ring $R=k[a,b,c,d,e,f]$, where $k$ is a field. We have that $a+b,f+e$ is a regular sequence on $R/I$ by \cite[Theorem 3.11]{FHM}. It turns out that $\depth R/I=2$, while $\depth R/I^2=1$. Notice that $f+e$ is not regular on $R/I^2$ and that $de$ and $d^2$ divide minimal generators of $I$. However, $a+b$ is regular on $R/I^2$, even though $bc,cd$, and $bd$ divide minimal generators of $I$. This means  that our condition in Theorem~\ref{sum 2 regular on power} is not a necessary condition.

A slight modification of this example gives a different result. Let $I=(a^2b,bc^2,bd,cd,de,ef)$ and now notice that $a+b$ is not regular on $R/I^2$, but $f+e$ is regular on $R/I^2$. The latter follows from Theorem~\ref{sum 2 regular on power}.
\end{example}

\begin{example}
Let $I$ be the edge ideal of a path of length $6$, depicted below, on the letters $a, \ldots, g$ in the polynomial ring $R=k[a,b,c,d,e,f,g]$, where $k$ is a field.
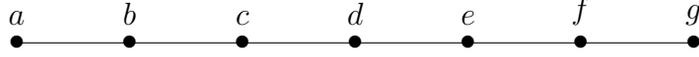
\begin{center}
	\begin{tikzpicture}
		\tikzstyle{point}=[inner sep=0pt]
		\node (a)[point,label=above:$a$] at (-2,0) {$\bullet$};
		\node (b)[point,label=above:$b$] at (-0.5,0) {$\bullet$};
		\node (c)[point,label=above:$c$] at (1,0) {$\bullet$};
		\node (d)[point,label=above:$d$] at (2.5,0) {$\bullet$};
		\node (e)[point,label=above:$e$] at (4,0) {$\bullet$};
		\node (f)[point,label=above:$f$] at (5.5,0) {$\bullet$};
		\node (g)[point,label=above:$g$] at (7,0) {$\bullet$};
		
		\draw (a.center) -- (b.center);
		\draw (b.center) -- (c.center);
		\draw (c.center) -- (d.center);
		\draw (d.center) -- (e.center);
		\draw (e.center) -- (f.center);
		\draw (f.center) -- (g.center);
	\end{tikzpicture}
\end{center}
Using Macaulay~2~\cite{M2} we can see that $\depth R/I=3$ and $\depth R/I^2=2$. Theorem~\ref{sum 2 regular on power} shows that $a+b$ and $g+f$ are both regular elements on $R/I^2$. One can also verify that $a+b, g+f$ is a regular sequence on $R/I^2$.

On the other hand if $I$ is the edge ideal of a path of length $3$ on the letters $a, \ldots, d$, even though the elements $a+b$ and $d+c$ are regular on $R/I^2$, they do not form a regular sequence on $R/I^2$. The reason is that $b,c \mid bc\in I$, see also Theorem~\ref{sequences for square} below.
\end{example}

We conclude this section by showing that binomial forms satisfying the conditions of Theorem~\ref{sum 2 regular on power}, whose supports are disjoint, can be combined to get a regular sequence on $R/I^2$, provided that there is sufficient distance between them, as was seen in the prior examples.

\begin{theorem} \label{sequences for square}
Let $I$ be a monomial ideal and suppose that $\{b_{1,0}, b_{1,1}\}, \ldots, \{b_{s,0}, b_{s,1}\}$ are sets of disjoint variables that satisfy the conditions of Theorem~\ref{sum 2 regular on power}. Moreover, suppose that for all $1\le j <r \le s$, there does not exist  $M\in \G(I)$ such that $b_{j,1} b_{r,1} \mid M$. Let $f_i=b_{i,0}+b_{i,1}$ for $1\le i \le s$. Then $f_1, \ldots, f_s$ is a regular sequence on $R/I^2$ and $\depth(R/I^2) \ge s$.
\end{theorem}

\begin{proof} Let $J = \ini(I^2, f_1, \dots, f_{s-1})$.
By \cite[Corollary~2.6, Lemma~2.4]{FHM} and induction, it suffices to show that $f_s$ is regular on $R/J$. We shall fix some notation. For any monomial $M \in R$, set
$$\widehat{M}=\dfrac{b_{1,1}^{d_{b_{1,0}}(M)} \cdots b_{s-1,1}^{d_{b_{s-1,0}}(M)} M}{b_{1,0}^{d_{b_{1,0}}(M)} \cdots b_{s-1,0}^{d_{b_{s-1,0}}(M)}}.$$

It follows from \cite[Lemma~2.3]{FHM} that
\begin{eqnarray*}
J&=&\langle b_{1,0}, \ldots, b_{s-1,0}, \widehat{N_j} \mid N_j \in \G(I^2)\rangle\\
&=&\langle b_{1,0}, \ldots, b_{s-1,0}, (\ini(I,f_1, \ldots, f_{s-1}))^2\rangle.
\end{eqnarray*}
Let $R' = R/(b_{1,0}, \dots, b_{s-1,0})$. To show that $f_s$ is regular on $R/J$ it suffices to show that $\{b_{s,0}, b_{s,1}\}$ satisfy the conditions of Theorem~\ref{sum 2 regular on power} for the image of $H=\ini(I,f_1, \ldots, f_{s-1})$ in $R'$. By abuse of notation, in the rest of the proof, for a monomial $M$ in $R$ we shall use $M$ to denote also the image of $M$ in $R'$ if $M$ does not contain any of the variables $\{b_{1,0}, \dots, b_{s-1,0}\}$.

Notice that since $d_{b_{s,0}}(I), d_{b_{s,1}}(I)\le 1$ and $\{b_{s,0}, b_{s,1}\} \cap \{b_{1,0}, b_{1,1}, \ldots, b_{s-1,0}, b_{s-1,1}\}=\emptyset$, we have $d_{b_{s,0}}(H) \le 1$ and $d_{b_{s,1}}(H) \le 1$, by \cite[Lemma~3.7]{FHM}. Also, if $b_{s,0}$ divides the image of $M$ in $R'$, then $b_{s,1}$ divides the image of $M$ as well.

It remains to establish the last condition of Theorem~\ref{sum 2 regular on power} for $\{b_{s,0},b_{s,1}\}$. Suppose, by contradiction, that there exist variables $c,d$ such that $b_{s,1}c \mid \widehat{M_1}$, $b_{s,1}d \mid \widehat{M_2}$, and $cd \mid \widehat{M_3}$, where $M_1, M_2, M_3$ are minimal monomial generators of $I$. If $\{c,d\} \cap \{ b_{1,1}, \ldots, b_{s-1, 1}\}=\emptyset$, then $c\mid M_1, M_3$, $d \mid M_2, M_3$, and $b_{s,1}\mid M_1,M_2$, contradicting the assumptions on $\{b_{s,0}, b_{s,1}\}$.  Suppose instead that one of $c$ or $d$ coincides with $b_{i,1}$ for some $i \le s-1$. Without loss of generality suppose $c=b_{i,1}$ for some $i\le s-1$. Then since $b_{i, 1} \mid \widehat{M_1}$, then either $b_{i,1} \mid M_1$ or $b_{i,0} \mid M_1$. In the latter case, we also have $b_{i,1} \mid M_1$, by the hypotheses (condition $(\star)$). Therefore, $b_{s,1},b_{i,1} \mid M_1$, contrary to our assumptions. Thus no such variables $c,d$ can exist, showing that, by Theorem~\ref{sum 2 regular on power}, $f_s$ is regular on $R/J$.
\end{proof}

\begin{corollary}\label{leaves bound}
	If $I$ is the edge ideal of a graph $G$ and $\{c_1, \ldots, c_r\}$ is a set of vertices in $G$ that are leaves such that $d(c_i, c_j) \ge 4$ for $i \neq j$, then $\depth(R/I^2) \ge r$.
\end{corollary}

\begin{proof}
	This follows immediately from Theorem~\ref{sequences for square} after noting that if $b_i$ is the unique neighbor of $c_i$, then $b_ib_j$ is not an edge of $G$ for all $i, j$. Thus, $c_1+b_1, c_2+b_2, \ldots , c_r+b_r$ is a regular sequence on $R/I^2$ by Theorem~\ref{sequences for square}.
	\end{proof}

The next example shows how the bound in Corollary~\ref{leaves bound} can be sharp.

\begin{example}
Let $I=(x_1x_2, x_2x_3,x_3x_4,x_4x_5, x_4x_6,x_6x_7,x_7x_8)$ be the edge ideal of the graph below in the polynomial ring $R=\mathbb{Q}[x_1, \ldots, x_8]$.

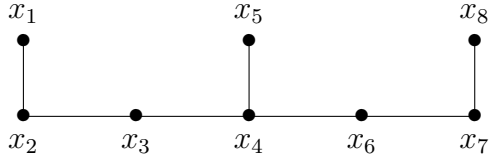
\begin{center}
	\begin{tikzpicture}
		\tikzstyle{point}=[inner sep=0pt]
		\node (a)[point,label=above:$x_1$] at (-2,1) {$\bullet$};
		\node (b)[point,label=below:$x_2$] at (-2,0) {$\bullet$};
		\node (c)[point,label=below:$x_3$] at (-0.5,0) {$\bullet$};
		\node (d)[point,label=below:$x_4$] at (1,0) {$\bullet$};
		\node (e)[point,label=above:$x_5$] at (1,1) {$\bullet$};
		\node (f)[point,label=below:$x_6$] at (2.5,0) {$\bullet$};
		\node (g)[point,label=below:$x_7$] at (4,0) {$\bullet$};
		\node (h)[point,label=above:$x_8$] at (4,1) {$\bullet$};
		
		\draw (a.center) -- (b.center);
		\draw (b.center) -- (c.center);
		\draw (c.center) -- (d.center);
		\draw (d.center) -- (e.center);
		\draw (d.center) -- (f.center);
		\draw (f.center) -- (g.center);
		\draw (g.center) -- (h.center);
	\end{tikzpicture}
\end{center}
Notice that by \cite[Theorem~3.11]{FHM} the sequence $x_1+x_2, x_5+x_4, x_8+x_7$ is regular on $R/I$ and by Corollary~\ref{leaves bound} and its proof, the same sequence is regular on $R/I^2$. One can confirm with Macaulay~2~\cite{M2} that $\depth R/I=\depth R/I^2=3$.
\end{example}

\section{Initially regular sequences of trinomial linear sums} \label{sec.tri}

In this section, we focus on linear sums consisting of three elements, trinomial linear sums, and investigate when sequences of such linear sums are initially regular with respect to the square of a given monomial ideal.

We begin by examining the initial ideal of a monomial ideal $I$ and a trinomial linear form.  We will use notation similar to that used in previous sections. Particularly, for fixed variables $b_0$ and $b_1$ in the polynomial ring $R$, and $M \in \G(I)$, set
 $$\widehat{M}=\dfrac{b_1^{d_{b_0}(M)}M}{b_0^{d_{b_0}(M)}}.$$
 For any $n\in \NN$ we use the notation $[n]$ to denote the set of integers from $1$ to $n$, i.e. $[n]=\{1, \ldots, n\}$.

\begin{lemma}\label{I, a+b+c, general}
Let $I$ be a monomial ideal. Let $b_0,b_1,b_2$ be variables satisfying condition $(\star)$. Furthermore, assume that $b_1b_2$ does not divide any $M\in\G(I)$.
If $>$ is a term order such that $b_0>b_1>b_2$, then
$$\ini(I,b_0+b_1+b_2)=\langle b_0, \widehat{M},\lcm(X,M')b_2^2\mid  M, b_1X, b_0b_2M'\in \G(I)\rangle.$$
\end{lemma}

\begin{proof}

	Let $\{M_{11}, \dots, M_{1i_1}\}$ and $\{M_{21}, \dots, M_{2i_2}\}$ be the monomial generators of $I$ that are divisible by $b_0b_1$ and $b_0b_2$, respectively. Set $M_{1j}' = M_{1j}/b_0b_1$ and $M_{2j}' = M_{2j}/b_0b_2$. Let $\{M_3, \dots, M_p\}$ be the remaining monomial generators of $I$.

Let $f=b_0+b_1+b_2$ and for $r\in[2]$, $i\in[i_r]$, and $j\in [i_2]$ we fix the following notation
$$p_{ri}=M_{ri}'b_r(b_1+b_2) \quad \text{ and } \quad q_{jX}=\lcm(X,M_{2j}')b_2^2,$$ where $b_1X\in \mathcal{G}(I)$.

We claim that $$G=\{M_{ri}, M_k, f, p_{ri}, q_{jX} ~\big|~  3\le k\le p, r\in[2],  i\in [i_r], j\in [i_2]\}$$ is a Gr\"{o}bner basis for $(I,f)$. Note that if $b_0 \mid X$, then $q_{jX}$ is redundant in the set $G$, so we may assume $b_0 \nmid X$ from now on.

Observe that for $r\in [2]$ and $i\in [i_r]$, we have
$$S(M_{ri}, f)=\frac{\lcm(b_0b_rM_{ri}', b_0)}{b_0}(b_1+b_2)=b_rM_{ri}'(b_1+b_2)=p_{ri}\in G.$$
Also, for any $M=b_1X\in \G(I)$ and $j\in [i_2]$, by Lemma~\ref{S-resultant tricks}~(b) we have
$$S(M, p_{2j})=S(b_1X, M_{2j}'b_2(b_1+b_2))=\frac{\lcm(b_1X,M_{2j'}b_2)}{b_1}b_2.$$
Since $b_1b_2$ does not divide any $N\in\G(I)$, it follows that
$$S(M, p_{2j})=\lcm(X, M_{2j}'b_2)b_2=\lcm(X, M_{2j}')b_2^2=q_{jX}.$$
Thus, all elements in $G$ indeed appear as we proceed through Buchberger's algorithm. It remains to show that all  the $S$-resultants among elements of $G$ reduce to $0$ modulo $G$.

Notice that for any monomials $M,N\in\G(I)$, $S(M,N)=0$ and $S(M, f)$ is redundant unless $b_0\mid M$. Furthermore, if $b_0 \mid M$, then $M=M_{ri}$, for some $r\in [2]$ and $i\in[i_r]$, and, as seen above, $S(M_{ri}, f)=p_{ri}$ reduces to 0 modulo $G$. Since $q_{jX}$ is monomial, $S(q_{jX}, f)$ is redundant. Hence, it remains to consider $S$-resultants of the following forms: (a) $S(M, p_{ri})$, where $M$ is monomial in $G$,  (b) $S(p_{ri}, p_{kj})$, and (c) $S(p_{ri}, f)$, with $r,i,k,j\in [2]$.

{\textbf{\textit{Type} (a):}} $S(M,p_{ri})$. For any monomial $M$ in $G$, we have that
$$S(M, p_{ri})=\frac{\lcm(M, M_{ri}'b_rb_1)}{b_1}b_2,$$
which is redundant unless $b_1\mid M$. If $M = b_1X\in \G(I)$ and $r=2$ then, as seen above, $S(b_1X, p_{2i})=q_{iX}$, reducing to 0 modulo $G$. If $M = b_1X$ and $r=1$, then
$$S(M, p_{1i})=\frac{\lcm(b_1X,M_{1i}'b_1^2)}{b_1}b_2,$$
which is a multiple of $b_1X$ and so reduces to 0 modulo $G$ as well. If $M = q_{jX}$, then  $S(M,p_{ri})$ becomes redundant since $b_1 \nmid q_{jX}$.

{\textbf{\textit{Type} (b):}} $S(p_{ri},p_{kj})$. By Lemma~\ref{S-resultant tricks}~(d) we have $S(p_{ri}, p_{kj})=0$.

{\textbf{\textit{Type} (c):}} $S(p_{ri},f)$. By Lemma~\ref{S-resultant tricks}~(c) we have $$S(p_{ri}, f)=b_0M_{ri}'b_rb_2-b_rM_{ri}'b_1(b_1+b_2)=M_{ri}b_2-b_1p_{ri},$$ which clearly reduces to $0$ modulo $G$, and the proof is complete.
\end{proof}

The next lemma focuses particularly on the square of a monomial ideal and a linear sum of three variables. This is the main technical part of our work in this section.

\begin{lemma}\label{ini I^2 a+b+c, general}
Let $I$ be a monomial ideal  in a polynomial ring $R$. Let $b_0,b_1,b_2$ be variables such that the following conditions hold:
\begin{enumerate}[label=$(\alph*)$]
\item $d_{b_i}(I) \le 1$ for $0 \le i \le 2$;
\item the only monomial generators in $I$ which $b_0$ divides are $b_0b_1$ and $b_0b_2$;
\item $b_1b_2 \nmid M$ for all $M \in \G(I)$;
\item there do not exist any variables $z_1, z_2$ such that $b_1z_1, b_1z_2$ and $z_1z_2$ divide monomial generators of $I$;
\item there do not exist any variables $x_1, x_2$, such that $b_1x_1, x_1x_2, x_2b_2$ divide monomial generators of $I$.
\end{enumerate}
Then, we have $$\ini(I^2, b_0+b_1+b_2)=\langle b_0,(\ini(I, b_0+b_1+b_2))^2 \rangle.$$
\end{lemma}

\begin{proof}
Let $k$ be the ground field and let $R_2=k[b_0,b_1,b_2]$. Let $R_1$ be the polynomial ring in the remaining variables; that is, $R=R_1[b_0,b_1,b_2]$. Let $f=b_0+b_1+b_2$. Let  $\G(I)=\{M_1, \dots, M_p\}$ and suppose, without loss of generality, that $M_1=b_0b_1$, $M_2=b_0b_2$, and $b_0\nmid M_i$ for all $i\ge 3$.

 For $r\in[2]$ set $p_{r}=b_r(b_1+b_2)$.
By Lemma~\ref{I, a+b+c, general},
$$G=\{M_i,p_{r}, Xb_2^2, f ~\big|~ i\in[p], r\in [2], b_1X\in \G(I)\}$$
is a Gr\"{o}bner basis for $(I,f)$.

\medskip
\noindent{\bf Claim:} A Gr\"{o}bner basis for $(I^2, f)$ is
$$G'=\{M_{i}M_j,  p_{r}p_{v}, XYb_2^4, M_ip_{r}, M_{i}Xb_2^2, Xb_2^2p_{r}, f ~\big|~  i,j\in[p], r,v\in[2], b_1X, b_1Y\in \G(I) \}.$$

\medskip

We note here that if $b_0 \mid X$, then the terms $M_iXb_2^2$ and $Xb_2^2p_r$ are multiples of other terms in $G'$. Moreover, if $b_0\mid X$ or $b_0 \mid Y$ and $b_0$ does not divide both, then we may assume without loss of generality that $b_0\mid X$ and $b_0 \nmid Y$. In this case, the term $XYb_2^4$ is a multiple of either $M_2Yb_2^2$. If $b_0\mid X, Y$, then the term $XYb_2^4$ is a multiple of  $M_1M_2$.  Hence we may assume that $b_0 \nmid X, Y$.

\noindent{\bf{Proof of Claim:}}  The proof proceeds in two steps.

\medskip
\noindent{\bf Step 1:} Starting with the generating set $\{M_iM_j, f\mid i,j\in[p]\}$ that generates $(I^2,f)$, we shall show that all elements in $G'$ arise as $S$-resultants when implementing Buchberger's algorithm, and thus are in the desired Gr\"{o}bner basis.

Following the algorithm, set $G_0=\{M_iM_j, f\mid i,j\in[p]\}$. For $r \in [2]$ and $i\in [p]$, we have
$$S(M_rM_i,f)=S(b_0b_rM_i,f)=\frac{\lcm(b_0b_rM_i,b_0)}{b_0}(b_1+b_2)=b_rM_i(b_1+b_2)=M_ip_r.$$
Since $M_ip_r$ does not reduce modulo $G_0$, then  $M_ip_r$ must be added to the set $G_0$ per the algorithm. Set
$$G_1=\{M_iM_j, M_ip_r, f\mid i,j\in[p]\}.$$ For $r,v\in [2]$, we have
\begin{eqnarray*}
S(M_rp_v, f)&=&\frac{\lcm(b_0b_rb_vb_1,b_0)}{b_1}b_2-\frac{\lcm(b_0b_rb_vb_1, b_0)}{b_0}(b_1+b_2)= b_rb_vb_0b_2-b_rb_vb_1(b_1+b_2).
\end{eqnarray*}
This can be further reduced by $b_rb_vb_2(b_0+b_1+b_2)$ to $$-b_rb_vb_2(b_1+b_2)-b_rb_vb_1(b_1+b_2)=-b_rb_v(b_1+b_2)^2=-p_{r}p_{v}.$$
Hence,
\begin{eqnarray}\label{-prpv}
S(M_rp_v,f) \text{ reduces to } p_rp_v.
\end{eqnarray}
Thus $p_rp_v$ must be added to the set. Set
$$G_2=\{M_iM_j, p_rp_v, M_ip_r,  f\mid i,j\in[p]\}.$$
Now, let $M=b_1X, N=b_1Y \in \G(I)$ and let $M_i\in \G(I)$. Then,
$$S(MM_i, M_ip_2)=S(b_1XM_i, M_ib_2(b_1+b_2))=\frac{\lcm(b_1XM_i, M_ib_1b_2)}{M_ib_1b_2}M_ib_2^2=M_iXb_2^2,$$
 which does not reduce modulo $G_2$. Hence $M_iXb_2^2$ must be added to $G_2$.
Let $$G_3=\{M_iM_j, p_rp_v,  M_ip_r, M_iXb_2^2, f\mid i,j\in[p]\}.$$
Now for  $M=b_1X, N=b_1Y \in \G(I)$ we have

$$S(MN, p_2^2)=S(b_1^2XY,b_2^2(b_1+b_2)^2)=\frac{\lcm(b_1^2XY, b_1^2b_2^2)}{b_1^2b_2^2}b_2^2(2b_1b_2+b_2^2)=XY(2b_1b_2^3+b_2^4),$$
which reduces to $XYb_2^4$, since $b_1XYb_2^2\in G_3$. Hence $M_iXb_2^2$ and $XYb_2^4$ must be added to the set. Thus we have
 $$G_4=\{M_iM_j, p_rp_v, XYb_2^4, M_ip_r, M_iXb_2^2, f\mid i,j\in[p]\}.$$
Finally  by Lemma~\ref{S-resultant tricks}~(b) we have,
$$S(M_rXb_2^2,f)=\frac{\lcm(b_0b_rXb_2^2,b_0)}{b_0}(b_1+b_2)=Xb_2^2b_r(b_1+b_2)=Xb_2^2p_r.$$
Adding $Xb_2^2p_r$ to the set yields
 $$G'=\{M_iM_j, p_rp_v, XYb_2^4, M_ip_r, M_iXb_2^2, Xb_2^2p_r, f\mid i,j\in[p]\}.$$

\medskip

\noindent{\bf Step 2:} In order to show that $G'$ is a Gr\"{o}bner basis, we must show that all the $S$-resultants of pairs of elements in $G'$ reduce to $0$ modulo $G'$. That is, we must show that the algorithm has stopped.

We need to consider $S$-resultants of the following types, for $i,j\in [p], r,v\in[2]$ and $X, Y$ such that $b_1X, b_1Y\in \G(I)$,
\begin{enumerate}
\item $S(M_iM_j, -)$
\item $S(p_{r}p_{v}, -)$
\item $S(XYb_2^4,-)$
\item $S(M_ip_{r}, -)$
 \item $S(M_iXb_2^2, -)$
\item  $S(Xb_2^2p_{r}, -)$.
\end{enumerate}

We shall examine these types of $S$-resultants one by one and show that they all reduce to 0 modulo $G'$.

\smallskip

\noindent{\bf \textit{Type 1: $S(M_iM_j, -)$.}}
Since $XYb_2^4$ and $M_iXb_2^2$ are monomial, there are four remaining pairings to consider. Observe that for a monomial $N$ and a polynomial $g$ we have, by  Lemma~\ref{S-resultant tricks}~(b),
\begin{eqnarray*}
S(M_iM_j,Ng)=\frac{\lcm(M_iM_j, N\ini(g))}{\ini(g)}(g-\ini(g)),\\
\end{eqnarray*}
and we we may assume that $\gcd(\ini(g), M_iM_j)\neq 1$.

\begin{itemize}
\item  $S(M_rM_i, f)=M_ip_r$ for $r\in [2]$, as seen above.

\item $S(M_iM_j, p_{k}p_{r})=S(M_iM_j,b_kb_r(b_1+b_2)^2)=\frac{\lcm(M_iM_j, b_kb_rb_1^2)}{b_1^2}(2b_1b_2+b_2^2)$. This is redundant modulo $G'$ if $\gcd(M_iM_j,b_1^2) = 1$. Assume that $\gcd(M_iM_j, b_1^2)\neq 1$. Without loss of generality we may assume that $M_i=b_1X$ and either $b_1\nmid M_j$ or $M_j=b_1Y$.

If $b_1\nmid M_j$ then, since $\gcd(b_kb_rb_1^2, X)=1$ by assumptions $($a$)$ and $($c$)$ and
$$S(M_iM_j, p_{k}p_{r})=\frac{\lcm(b_1XM_j, b_kb_rb_1^2)}{b_1^2}b_2(2b_1+b_2)=X\lcm(M_j, b_kb_r)b_2(2b_1+b_2).$$
Notice that $b_kb_r=b_1^2, b_1b_2$, or $b_2^2$. In any case, at most one of $b_k$ and $b_r$ divides $M_j$, so we may assume that $b_k\nmid M_j$. Then, $S(M_iM_j, p_{k}p_{r})$ is a multiple of $XM_jb_kb_2$, which is either $(b_1X)M_jb_2$ or $XM_jb_2^2$, and both reduce to 0 modulo $G'$.

If $M_i=b_1X$ and $M_j=b_1Y$, then
$$S(M_iM_j, p_{k}p_{r})=\lcm(XY, b_kb_r)b_2(2b_1+b_2)=XYb_kb_rb_2(2b_1+b_2).$$
Again, $b_kb_r$ can be either $b_1^2, b_1b_2,$ or $b_2^2$. In any of these cases, $S(M_iM_j,p_{k}p_r)$ reduces to 0 modulo $G'$, since $(b_1X)(b_1Y)$, $(b_1X)Yb_2^2$, and $XYb_2^4$ are all in $G'$.

\item $S(M_iM_j, M_kp_{r})=\frac{\lcm(M_iM_j, M_kb_rb_1)}{b_1}b_2$ for any $k\in[p]$. As before, we may assume that $\gcd(M_iM_j, b_1)\neq 1$. Without loss of generality, suppose that $M_i=b_1X$. Then,
$$S(M_iM_j, M_kp_{r})=\lcm(XM_j,M_kb_r)b_2.$$
If $b_r\nmid M_j$, then this resultant is a multiple of $XM_jb_rb_2$, which is redundant. If $M_j=b_rZ$ for some $Z$, then $S(M_iM_j, M_kp_{r})=\lcm(XZb_r, M_kb_r)b_2$. To this end, we consider two possibilities.

 If $\gcd(Z, M_k)=1$, then this resultant is a multiple of $b_rZM_k = M_jM_k$, which is redundant. Suppose then that $M_k=Z_1W$ for some $W$ and $Z_1\mid Z$ with $\gcd(Z,W)=1$. If $\gcd(W,X)=1$, then we have
 $$S(M_iM_j, M_kp_{r})=\lcm(XZb_r, ZWb_r)b_2=XZWb_rb_2=(ZW)(Xb_rb_2),$$ which is redundant since $M_iM_k \mid (ZW)(b_1X)$ and $M_kXb_2^2 \mid (ZW)Xb_2^2$. If $\gcd(W,X)\neq 1$, then there exists $z_2$ such that $z_2\mid W$ and $z_2\mid X$.
Then  $b_1z_2\mid M_i$, and for $z_1 \mid Z_1$, $z_1z_2\mid M_k$ and $b_rz_1\mid M_j$. If $b_r=b_1$, this contradicts assumption (d) and if $b_r=b_2$ this contradicts assumption (e).

\item $S(M_iM_j, Xb_2^2p_{r})=\frac{\lcm(M_iM_j, Xb_2^2b_rb_1)}{b_1}b_2$. We may assume that $\gcd(M_iM_j, b_1)\neq 1$ and $M_i=b_1Y$.
Then $S(M_iM_j, Xb_2^2p_{r})= \lcm(YM_j, Xb_2^2b_r)b_2$, which is a multiple of $YM_jb_2^2\in G'$, since $b_2^2\nmid M_j$.  This completes the arguments for type 1 $S$-resultants.
\end{itemize}

\smallskip

\noindent{\bf \textit{Type 2: $S(p_{r}p_{v}, -)$.}} We need to consider six different choices of the second argument of the $S$-resultant. Note that $r,v \in [2]$. In what follows, let $s,k \in [2].$

\begin{itemize}
\item $S(p_{r}p_{v}, p_{s}p_{k})=S(b_rb_v(b_1+b_2)^2, b_sb_k(b_1+b_2)^2)=0$, by Lemma~\ref{S-resultant tricks}~(d).

\item $S(p_{r}p_{v}, XYb_2^4)=\frac{\lcm(b_rb_vb_1^2, XYb_2^4)}{b_1^2}(2b_1b_2+b_2^2)$, by Lemma~\ref{S-resultant tricks}~(b),  which is clearly a multiple of $XYb_2^4 \in G'$.

\item $S(p_{r}p_{v}, M_ip_{s})=S(b_rb_v(b_1+b_2)^2, M_ib_s(b_1+b_2))=\frac{\lcm(b_rb_vb_1^2, M_ib_sb_1)}{b_1^2}b_2(b_1+b_2)$. This is in fact equal to $\frac{\lcm(b_rb_vb_1^2, M_ib_sb_1)}{b_1^2}p_2$.  If $b_1\nmid M_i$, then this is a multiple of $M_i p_2 \in G'$. If $b_1\mid M_i$, then $M_i=b_1X$. If $r=1$ or $v=1$ or $s=1$, then the $S$-resultant is a multiple of $M_ip_2$ again. If instead $r=s=v=2$, then the $S$-resultant becomes $b_2^2Xp_2\in G'$.

\item $S(p_{r}p_{v}, M_iXb_2^2)=\frac{\lcm(b_rb_vb_1^2, M_iXb_2^2)}{b_1^2}(2b_1b_2+b_2^2)$, by Lemma~\ref{S-resultant tricks}~(b). We again can assume that $\gcd(b_1^2, M_iXb_2^2)\neq 1$ and $M_i=b_1Y$. Then,  $$S(p_{r}p_{v}, M_iXb_2^2)=\lcm(b_rb_v, XYb_2^2)b_2(2b_1+b_2),$$ which is a multiple of $YXb_2^3(2b_1+b_2)$. The latter reduces to $0$ modulo $G'$, since  $b_1Y\in I$ and $YXb_2^4 \in G'$.

\item $S(p_{r}p_{v},Xb_2^2p_{k})=S(b_rb_v(b_1+b_2)^2, Xb_2^2b_k(b_1+b_2))=\frac{\lcm(b_rb_v, Xb_2^2b_k)}{b_1^2}b_2(b_1+b_2)$. This is equal to $\frac{\lcm(b_rb_v, Xb_2^2b_k)}{b_1^2}p_{2}$, which is clearly a multiple of $Xb_2^2p_2\in G'$.

\item  Finally,
 \begin{eqnarray*}
 S(p_{r}p_{v},f)&=&S(b_rb_v(b_1+b_2)^2, b_0+b_1+b_2)=b_0b_rb_v(2b_1b_2+b_2^2)-b_rb_vb_1^2(b_1+b_2)\\
 &=&b_0b_2b_rb_v(b_1+b_2)+b_0b_2b_rb_vb_1-b_rb_vb_1^2(b_1+b_2)\\
 &=&b_0b_2b_rp_{v}+b_0b_2b_rb_vb_1-b_rb_vb_1^2(b_1+b_2),
 \end{eqnarray*}
which reduces to $b_0b_2b_rb_vb_1-b_rb_vb_1^2(b_1+b_2)$. This can be further reduced by $b_rb_vb_1b_2(b_0+b_1+b_2)$ to $-b_rb_vb_1b_2(b_1+b_2)-b_rb_vb_1^2(b_1+b_2)=-b_rb_vb_1(b_1+b_2)^2=-b_1p_{r}p_{v}$.
\end{itemize}

\smallskip

\noindent{\bf \textit{Type 3: $S(XYb_2^4,-)$.}} For this type of $S$-resultants, there are only three remaining choices for the second argument to consider: $S(XYb_2^4,Zb_2^2p_{r})$, $S(XYb_2^4,M_ip_{r})$, and $S(XYb_2^4,f)$.

Let $N$ represent either $Zb_2^2$ or $M_i$. Then $S(XYb_2^4,Np_{r}) = \frac{\lcm(XYb_2^4,Nb_rb_1)}{b_1}{b_2}$, which is a multiple of $XYb_2^4 \in G'$. Thus, both $S(XYb_2^4,Zb_2^2p_{r})$ and  $S(XYb_2^4,M_ip_{r})$ reduce to 0 modulo $G'$. On the other hand, $S(XYb_2^4, f)$ is redundant since $b_0\nmid XYb_2^4$. 
\smallskip

\noindent{\bf \textit{Type 4: $S(M_ip_{r}, -)$.}} There are four choices remaining to be considered.

\begin{itemize}
	\item $S(M_ip_{r}, M_jp_{v})=S(M_ip_{r}, Xb_2^2p_{v})=0$, by Lemma~\ref{S-resultant tricks}~(d).
	
	 \item $S(M_ip_{r}, M_jXb_2^2)=\frac{\lcm(M_ib_rb_1, M_jXb_2^2)}{b_1}b_2$. As before, we may assume that  $M_j=b_1Y$ and thus
	 $$S(M_ip_{r}, M_jXb_2^2)=\frac{\lcm(M_ib_rb_1, b_1YXb_2^2)}{b_1}b_2=\lcm(M_ib_r, YXb_2^2)b_2.$$
	 If $b_r=b_1$, then this resultant is a multiple of $(b_1X)Yb_2^2$. If $b_r=b_2$, then this resultant becomes $\lcm(M_ib_2, XYb_2^2)b_2=\lcm(M_i,XYb_2)b_2^2$. If $\gcd(M_i, X)=1$, then this is a multiple of $M_iXb_2^2$. Hence we may assume that $\gcd(M_i,X)\neq 1$ and $\gcd(M_i,Y)\neq 1$.

Let $X_1=\gcd(M_i, X)$ and $Y_1=\gcd(M_i,Y)$. Suppose first that there exist variables $x_1,y_1$ such that $x_1\mid X_1$ and $y_1\mid Y_1$ with $x_1\neq y_1$. Then $b_1x_1\mid b_1X$, $b_1y_1\mid b_1Y$, and $x_1y_1\mid M_i$, a contradiction to condition~(d). Hence we may assume that there do not exist any such different $x_1, y_1$ as above. Then $X_1=Y_1$ and thus
the $S$-resultant  $\lcm(M_i, XYb_2)b_2^2$ is a multiple of $\frac{M_i}{X_1}XYb_2^2$ that we can rewrite as a multiple of $M_iXb_2^2$, since $X_1\mid Y$.
	
	\item $S(M_ip_{r}, f)=\frac{\lcm(M_ib_rb_1,b_0)}{b_1}b_2-\frac{\lcm(M_ib_rb_1, b_0)}{b_0}(b_1+b_2)$. If $b_0\nmid M_i$, then the first part is a multiple of $M_i$ and $b_0b_2$, whereas the second is a multiple of $M_ib_1(b_1+b_2)$. Suppose that $b_0\mid M_i$. Then, $M_i=b_0b_i$ and the $S$-resultant $S(M_ip_r,f)$ reduces to $p_ip_r$ as we showed in Equation~\eqref{-prpv}.
 \end{itemize}

 \smallskip

\noindent{\bf \textit{Type 5: $S(M_iXb_2^2, -)$.}} There are two choices remaining to be considered.

\begin{itemize}
	\item $S(M_iXb_2^2, Yb_2^2p_{r})=\frac{\lcm(Yb_2^2b_rb_1, M_iXb_2^2)}{b_1}b_2$. Again, we may assume that $M_i=b_1Z$. Since $b_r=b_1$ or $b_2$ and $b_1b_2 \nmid M_k$ for all $k\in [p]$, this $S$-resultant is a multiple of $b_rZXb_2^3$. It is easily seen that this is redundant.
 \item $S(M_iXb_2^2, f)=\frac{\lcm(M_iXb_2^2,b_0)}{b_0}(b_1+b_2)$. Again, we may assume that $M_i=b_0b_i$ and $i\in [2]$. Hence the $S$-resultant becomes $Xb_2^2b_i(b_1+b_2)=Xb_2^2p_{i} \in G'$.
\end{itemize}

\smallskip

\noindent{\bf \textit{Type 6: $S(Xb_2^2p_{i}, -)$.}} There are two choices remaining to be considered.

\begin{itemize}
	\item $S(Xb_2^2p_{r}, Yb_2^2p_{v})=0$,  by Lemma~\ref{S-resultant tricks}~(d).
	\item $S(Xb_2^2p_{r}, f)=\frac{\lcm(Xb_2^2b_rb_1, b_0)}{b_1}b_2-\frac{\lcm(Xb_2^2b_rb_1, b_0)}{b_0}(b_1+b_2)=Xb_2^2b_0b_rb_2-b_1Xb_2^2b_r(b_1+b_2)$, which is clearly redundant.
\end{itemize}

\medskip

Combining Steps 1 and 2 of the proof, we have established that $G'$ is a Gr\"{o}bner basis for $(I^2, f)$. The conclusion of Lemma~\ref{ini I^2 a+b+c, general} follows by observing that
$$\ini(I,f) = \langle b_0, M_i, b_1^2, b_1b_2, Xb_2^2 \mid i \ge 3, b_1X \in \G(I)\rangle$$
and $\ini(p_{1}^2)=b_1^4$, $\ini(p_{2}^2)=b_1^2b_2^2$, and $\ini(p_1p_2)=b_1^3b_2$.
\end{proof}

Lemma~\ref{ini I^2 a+b+c, general} can be applied repeatedly to obtain the next corollary.

\begin{corollary} \label{iterated a+b+c}
Let $I$ be a monomial ideal and let $t \in \NN$. Let $\{b_{i,0},b_{i,1}, b_{i,2}\}$ be disjoint sets of variables and let $>_{i}$  be a term order, for $1 \le i \le t$, such that $b_{i,0}>b_{i,1}>b_{i,2}$.
Suppose that for each $i$ the variables $b_{i,0}, b_{i,1}, b_{i,2}$ satisfy the hypotheses of Lemma~\ref{ini I^2 a+b+c, general}.
Let $I_i=\ini(I_{i-1}, b_{i,0}+b_{i,1}+b_{i,2})$, where $I_{0}=I$. Then $$\ini(\ldots (\ini( \ini(I^2, b_{1,0}+b_{1,1}+b_{1,2}), b_{2,0}+b_{2,1}+b_{2,2})),\ldots), b_{t,0}+b_{t,1}+b_{t,2})=\langle b_{1,0}, b_{2,0}, \ldots, b_{t,0}, I_{t}^2\rangle .$$

\end{corollary}

The following remark gives an interpretation of the assumptions in Lemma~\ref{ini I^2 a+b+c, general} when we specialize to edge ideals of graphs.

\begin{remark}
Let $I$ be the edge ideal of a graph and let $b_0,b_1,b_2 \in I$ be variables as in Lemma~\ref{ini I^2 a+b+c, general}. Then the assumptions (a)-(e) mean that $b_0, b_1, b_2$ do not form a triangle, $b_1$ is not on any triangle, and $b_0,b_1,b_2$ are not part of a pentagon. If these assumptions fail, then the statement of Lemma~\ref{ini I^2 a+b+c, general} no longer holds, as seen in the next example.
\end{remark}

\begin{example}
Let $I=(x_1x_2,x_2x_3,x_3x_4,x_4x_5,x_1x_5)$ be the edge ideal of a pentagon in the polynomial ring $R=\mathbb{Q}[x_1,x_2,x_3,x_4,x_5]$.
Using Macaulay~2~\cite{M2} we can verify that $\ini(I^2, x_1+x_2+x_5)\neq (x_1, (\ini(I,x_1+x_2+x_5))^2)$, where $>$ is a term order such that $x_1>x_2>x_5$.
In particular, $x_4x_5, x_3x_5^2\in \ini(I,x_1+x_2+x_5)$ by Lemma~\ref{I, a+b+c, general} and thus $x_3x_4x_5^3\in (x_1, (\ini(I,x_1+x_2+x_5))^2)$. However, $x_3x_4x_5^2$ is a minimal generator of $\ini(I^2, x_1+x_2+x_5)$, but $x_3x_4x_5^2 \not\in (x_1, (\ini(I,x_1+x_2+x_5))^2)$.
\end{example}

We now arrive to the main theorem of this section. For a monomial ideal $I$ and a variable $x$, we define
$$N(x) = \{M ~\big|~ xM \in \mathcal{G}(I)\}.$$

\begin{theorem} \label{sequence sum 3}
Let $I$ be a monomial ideal and let $s \in \NN$. Let $\{b_{i,0},b_{i,1}, b_{i,2}\}$ be disjoint sets of variables, and let $>$  be a term order such that $b_{i,0}>b_{i,1}>b_{i,2}$.
Suppose that for each $i$ the variables $b_{i,0}, b_{i,1}, b_{i,2}$ satisfy the hypotheses of Lemma~\ref{ini I^2 a+b+c, general}. In addition, assume $\{b_{i,0},b_{i,1}, b_{i,2}\} \cap N(b_{j,r})=\emptyset$ for $j\neq i$, $r\in[2]$. Let $f_i=b_{i,0}+b_{i,1}+b_{i,2}$. Then $f_1, \ldots, f_s$ is an initially regular sequence on $R/I^2$.
\end{theorem}

\begin{proof} Set $I_i = \ini(I_{i-1}, f_i)$, where $I_0 = I$.
By Corollary~\ref{iterated a+b+c}, it suffices to show that $f_i$ is regular on $R/(b_{1,0}, \ldots, b_{i-1,0}, I_{i-1}^2)$. By Lemma~\ref{ini I^2 a+b+c, general}, there exists an ideal $J$ such that
$$R/(b_{1,0}, \ldots, b_{i-1,0}, I_{i-1}^2)\cong R'/J^2,$$
where $R'$ is a polynomial ring such that $R=R'[b_{1,0}, \ldots, b_{i-1,0}]$. Let $\mathcal{G}(I)=\{M_1, \ldots, M_p\}$ and suppose without loss of generality that for all $j\in [i-1]$ we have $b_{j,0}\nmid M_1, \ldots, M_r$ and $b_{j,0}\mid M_i$ for all $r+1\le i\le p$. Then by Lemma~\ref{I, a+b+c, general} we have the following description for the ideal $J$:
\begin{eqnarray*}
 J&=&I+\langle b_{j,1}^2,  b_{j,1}b_{j,2}, Xb_{j,2}^2 \mid X\in N(b_{j,1}), j\in[i-1] \rangle\\
&=&\langle M_1,\ldots, M_r, b_{j,1}^2, b_{j,1}b_{j,2}, Xb_{j,2}^2 \mid  X\in N(b_{j,1}), j\in[i-1]\rangle.
\end{eqnarray*}

By \cite[Lemma~3.7]{FHM} we also know that $d_{b_{i,u}}(J)\le 1$ for $ 0\le u \le 2$. Polarizing $J$, we obtain
$$J^{\pol}=\langle M_i^{pol},  b_{j,1}b_{j,1}', b_{j,1}b_{j,2}, X^{pol}b_{j,2}b_{j,2}' ~\big|~ i\in [r], X\in N(b_{j,1}), j\in[i-1] \rangle \subseteq R'^{\pol}.$$
Let $H$ be the hypergraph corresponding to $J^{\pol}$.

Notice that by the hypotheses, $f_i$ is regular on $R'^{\pol} /J^{\pol}$ by \cite[Theorem 3.11]{FHM} and $\{b_{i,0}, b_{i,1}, b_{i,2}\}$ is not in any 2-saturating set of the hypergraph $H$ of $J^{\pol}$. Hence, it follows from \cite[Theorem~3.1]{TT} that $f_i$ is regular on $R'^{\pol}/(J^{\pol})^2$. This means that $f_i\not\in \Ass(R'^{\pol}/(J^{\pol})^2)$.

We claim that $f_i\not\in \Ass(R'/J^{2})$. Consider an arbitrary $Q\in \Ass(R'/J^2)$. By \cite[Corollary~2.6]{Far}, we have $Q=Q'^{\rm{depol}}$, where $Q'\in \Ass(R'^{\pol}/(J^{\pol})^2)$ and $Q'^{\rm{depol}}$ is the depolarization of $Q'$. However, $f_i\not \in Q'$ means that there exists a value $r$ such that $b_{i,r}\not \in Q'$. Since $b_{i,r}\not\in \{b_{j,1}',b_{j,2}'\}$ for all $j\in[i-1]$ and $d_{b_{i,r}}(J) =1$, we get $b_{i,r}\not \in Q$. Hence, $f_i\not \in Q$ and thus $f_i$ is regular on $R'/J^2$.
\end{proof}

The next corollary allows us to combine the results from Theorems~\ref{sequences for square}, and~\ref{sequence sum 3} to obtain longer initially regular sequences.
\begin{corollary}\label{combine}
Let $I$ be a monomial ideal and let $s, t\in \NN$. For for $i\in [s]$ and $j\in [t]$ let $\{a_{i,0}, a_{i,1}\}$, $\{b_{j,0}, b_{j,1}, b_{j,2}\}$ be disjoint sets of variables and let $>$ be a term order such that $a_{i,0}>a_{i,1}$ and $b_{j,0}>b_{j,1}>b_{j,2}$ for all $i\in [s]$ and $j\in [t]$. For every $i\in [s]$ and $j\in [t]$ suppose that $\{a_{i,0}, a_{i,1}\}$ satisfy the hypothesis of Theorem~\ref{sequences for square} and that $\{b_{j,0}, b_{j,1}, b_{j,2}\}$ satisfy the hypotheses of Theorem~\ref{sequence sum 3}. In addition, assume that $\{a_{i,0}, a_{i,1}\}\cap N(b_{j,r})=\emptyset$  for every $i\in [s]$, $j\in[t]$, and $r\in [2]$.  Let $f_i=a_{i,0}+a_{i,1}$ and $g_j=b_{j,0}+b_{j,1}+b_{j,2}$. Then any sequence obtained from selecting terms from $f_1, \ldots, f_s$ and $g_1, \ldots, g_t$ is an initially regular sequence on $R/I^2$. In particular, $f_1, \ldots, f_s, g_1, \ldots, g_t$ is an initially regular sequence on $R/I^2$ and $\depth R/I^2\ge s+t$.
\end{corollary}

The next example illustrates the above corollary.

\begin{example}
Let $R=\mathbb{Q}[x_1, \ldots, x_{13}]$ be a polynomial ring and let $$I=(x_1x_2, x_2x_3, x_3x_4, x_4x_5, x_5x_6, x_4x_7, x_7x_8, x_8x_9, x_9x_{10}, x_{10}x_{11}, x_{11}x_{12}, x_{12}x_{13})$$
be the edge ideal of the graph depicted below.

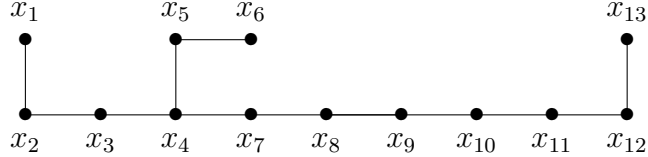
\begin{center}
	\begin{tikzpicture}
		\tikzstyle{point}=[inner sep=0pt]
		\node (a)[point,label=above:$x_1$] at (-2,1) {$\bullet$};
		\node (b)[point,label=below:$x_2$] at (-2,0) {$\bullet$};
		\node (c)[point,label=below:$x_3$] at (-1,0) {$\bullet$};
		\node (d)[point,label=below:$x_4$] at (0,0) {$\bullet$};
		\node (e)[point,label=above:$x_5$] at (0,1) {$\bullet$};
		\node (f)[point,label=above:$x_6$] at (1,1) {$\bullet$};
		\node (g)[point,label=below:$x_7$] at (1,0) {$\bullet$};
		\node (h)[point,label=below:$x_8$] at (2,0) {$\bullet$};
		\node (j)[point,label=below:$x_9$] at (3,0) {$\bullet$};
		\node (i)[point,label=below:$x_{10}$] at (4,0) {$\bullet$};
		\node (k)[point,label=below:$x_{11}$] at (5,0) {$\bullet$};
		\node (l)[point,label=below:$x_{12}$] at (6,0) {$\bullet$};
		\node (m)[point,label=above:$x_{13}$] at (6,1) {$\bullet$};

		\draw (a.center) -- (b.center);
		\draw (b.center) -- (c.center);
		\draw (c.center) -- (d.center);
		\draw (d.center) -- (e.center);
		\draw (e.center) -- (f.center);
		\draw (d.center) -- (g.center);
		\draw (g.center) -- (h.center);
		\draw (j.center) -- (h.center);
		\draw (h.center) -- (i.center);
		 \draw (i.center) -- (k.center);
		 \draw (k.center) -- (l.center);
		 \draw (l.center) -- (m.center);

	\end{tikzpicture}
\end{center}

Let $f_1=x_1+x_2$, $f_2=x_{13}+x_{12}$ and $g_1=x_5+x_6+x_4$, $g_2=x_9+x_8+x_{10}$. Then any combination of $f_1, f_2, g_1, g_2$ gives an initially regular sequence on $R/I^2$ with respect to any order $>$ such that $x_1>x_2$, $x_{13}>x_{12}$, $x_5>x_6>x_4$, and $x_9>x_8>x_{10}$. This can be also verified using Macaulay~2~\cite{M2}. Moreover, $\depth R/I^2=4$,  while $\depth R/I=5$. For the first power our regular elements can be chosen to be closer together and hence by \cite[Theorem~3.11]{FHM} the sequence $x_1+x_2, x_6+x_5, x_{13}+x_{12}, x_7+x_4+x_8, x_{10}+x_9+x_{11}$ is an initially regular sequence on $R/I$ with respect to any order $>$ such that $x_1>x_2, x_6>x_5, x_{13}>x_{12}, x_7>x_4>x_8, x_{10}>x_{9}>x_{11}$.
\end{example}

In order to extend any of our results one would need to analyze the initial ideal of a monomial ideal and a linear sum of arbitrary length as well as the initial ideal of the square of the ideal with an arbitrary linear sum. This can get quite complicated even in the case of edge ideals. We are able to compute $\ini(I, b_0+\ldots +b_t)$ for an edge ideal, but $\ini(I^2, b_0+\ldots +b_t)$ is difficult to describe. The following example gives an idea of how complicated $\ini(I, b_0+\ldots+b_t)$ is, even when $I$ is the edge ideal of a graph.

\begin{example}
	Let $R=\mathbb{Q}[b_0,b_1,b_2,b_3,b_4, b_5,x_1,x_2,x_3,x_4]$ be a polynomial ring and let
	$$I=(b_0b_1,b_0b_2,b_0b_3,b_0b_4,b_0b_5,b_1x_1,b_2x_2,b_3x_3,b_4x_4)$$
	be the edge ideal of the graph depicted below.
	
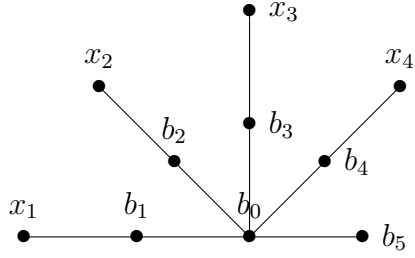
\begin{center}
		\begin{tikzpicture}
			\tikzstyle{point}=[inner sep=0pt]
			\node (a)[point,label=above:$b_0$] at (0,0) {$\bullet$};
			\node (b)[point,label=above:$b_1$] at (-1.5,0) {$\bullet$};
			\node (c)[point,label=above:$x_1$] at (-3,0) {$\bullet$};
			\node (d)[point,label=above:$b_2$] at (-1,1) {$\bullet$};
			\node (e)[point,label=above:$x_2$] at (-2,2) {$\bullet$};
			\node (f)[point,label=right:$b_3$] at (0,1.5) {$\bullet$};
			\node (g)[point,label=right:$x_3$] at (0,3) {$\bullet$};
			\node (h)[point,label=right:$b_4$] at (1,1) {$\bullet$};
			\node (i)[point,label=above:$x_4$] at (2,2) {$\bullet$};
			\node (j)[point,label=right:$b_5$] at (1.5,0) {$\bullet$};
			
			\draw (a.center) -- (b.center);
			\draw (b.center) -- (c.center);
			\draw (a.center) -- (d.center);
			\draw (d.center) -- (e.center);
			\draw (a.center) -- (f.center);
			\draw (f.center) -- (g.center);
			\draw (a.center) -- (h.center);
			\draw (h.center) -- (i.center);
			\draw (a.center) -- (j.center);
		\end{tikzpicture}
	\end{center}
Notice that $N(b_0)=\{b_1,b_2,b_3,b_4,b_5\}$ and $N(b_i)=\{x_i,b_0\}$ for all $i\in [4]$. Then using Macaulay 2~\cite{M2} we see that
	\begin{eqnarray*}
		\ini(I, b_0+b_1+b_2+b_3+b_4+b_5)&=&(b_0,b_1x_1,b_2x_2,b_3x_3,b_4x_4,b_1^2, b_1b_2,b_1b_3,b_1b_4,b_1b_5,\\
		&&x_1b_2(b_2,b_3,b_4,b_5),x_1x_2b_3(b_3,b_4,b_5),\\ && x_1x_2x_3b_4(b_4,b_5),x_1x_2x_3x_4b_5^2 ).
	\end{eqnarray*}
	
\end{example}

\section*{Conflict of Interest Statement}
On behalf of all authors, the corresponding author states that there is no conflict of interest.

\end{document}